\documentclass[11pt,a4paper,leqno,twoside, headinclude]{amsart}

{\providecommand{\noopsort}[1]{}
\usepackage[tracking=true]{microtype}
\DeclareMicrotypeSet*[tracking]{my}%
  { font = */*/*/sc/* }%
\SetTracking{ encoding = *, shape = sc }{ 45 }

\title[Positive curvature under symmetry]{Topological properties of positively curved manifolds with symmetry}
\def\titl{Topological properties of positively curved manifolds with symmetry}
\def\auth{Manuel Amann and Lee Kennard}
\date{June 3rd, 2012}

\subjclass[2010]{53C20 (Primary), 57N65 (Secondary)}
\keywords{\noindent positive curvature, symmetry, Euler characteristic, Betti numbers, Hopf conjecture, elliptic genus, symmetric spaces}
\thanks{The first author was supported by a grant of the German Research Foundation. The second author was partially supported by National Science Foundation Grant DMS-1045292.}

\author{\auth}

\usepackage[english]{babel}

\usepackage[colorlinks,pdftex, plainpages=false]{hyperref}
\hypersetup{pdftitle=\titl, pdfauthor=\auth, pdftoolbar=false,
plainpages=false, hyperindex=true, pdfdisplaydoctitle=true}

\hypersetup{colorlinks,%
citecolor=black,%
filecolor=black,%
linkcolor=black,%
urlcolor=black,%
pdftex}

\usepackage{color}

\usepackage{amsmath, amssymb, amscd, amsthm}
\usepackage{stmaryrd}
\usepackage{latexsym}
\usepackage[all]{xy}
\usepackage{pb-diagram}
\usepackage{rotating}
\usepackage{multicol}
\usepackage{lscape}
\usepackage{wasysym}
\usepackage{longtable}
\usepackage{enumerate}

\xyoption{all}

\newtheorem{theo}{Theorem}[section]
\newtheorem{main}{Theorem}
\newtheorem{maincor}[main]{Corollary}
\newtheorem*{main*}{Theorem}

\newtheorem*{mainprop*}{Proposition}

\newtheorem{mainconj}{Conjecture}

\newtheorem{prop}[theo]{Proposition}
\newtheorem{defi2}[theo]{Definition}
\newtheorem*{defi2*}{Definition}
\newenvironment{defi}{\begin{defi2}\normalfont}{\end{defi2}}
\newenvironment{defi*}{\begin{defi2*}\normalfont}{\end{defi2*}}

\newenvironment{defin*}[1]{\begin{defi2*}[#1]\normalfont}{\end{defi2*}}
\newtheorem{rem2}[theo]{Remark}
\newtheorem*{rem2*}{Remark}
\newenvironment{rem}{\begin{rem2}\normalfont}{\hfill$\boxbox$\end{rem2}}
\newenvironment{rem*}{\begin{rem2*}\normalfont}{\hfill$\boxbox$\end{rem2*}}

\newtheorem{lemma}[theo]{Lemma}

\newtheorem*{cor*}{Corollary}

\newtheorem*{conj*}{Conjecture}
\newtheorem*{theo*}{Theorem}
\newtheorem*{ques*}{Question}
\newtheorem*{mi2}{Main Idea}

\newtheorem{ex2}[theo]{Example}
\newenvironment{ex}{\begin{ex2}\normalfont}{\hfill$\boxbox$\end{ex2}}

\newtheorem{exer2}[theo]{Exercise}

\newtheorem{alg2}[theo]{Algorithm}

\newcommand{\cc}{{\mathbb{C}}}                                     
\newcommand{\hh}{{\mathbb{H}}}                                     
\newcommand{\pp}{{\mathbf{P}}}                                     
\newcommand{\s}{{\mathbb{S}}}                                      
\newcommand{\SO}{{\mathbf{SO}}}                                    
\newcommand{\U}{{\mathbf{U}}}                                      
\newcommand{\id}{{\operatorname{id}}}                              
\newcommand{\W}{{\operatorname{W}}}                                

\newcommand{\comment}[1]{}                                         
\newcommand{\ack}{\noindent\textbf{Acknowledgements. }}            

\newcommand{\thm}[1]{Theorem \ref{#1}}

\newcommand{\PROP}[1]{Proposition \ref{#1}}

\newcommand{\Z}{\mathbb{Z}}
\newcommand{\Q}{\mathbb{Q}}

\newcommand{\embedded}{\hookrightarrow}
\DeclareMathOperator{\cod}{cod}
\newcommand{\ceil}[1]{\left\lceil #1 \right\rceil}
\newcommand{\floor}[1]{\left\lfloor #1 \right\rfloor}
\newcommand{\pfrac}[2]{\left(\frac{#1}{#2}\right)}
\newcommand{\of}[1]{\left(#1\right)}

\newcommand{\st}{~|~}

\DeclareMathOperator{\beven}{b_{\mathrm{even}}}

\begin{document}


\begin{abstract}
Manifolds admitting positive sectional curvature are conjectured to have rigid homotopical structure and, in particular, comparatively small Euler charateristics.

In this article, we obtain upper bounds for the Euler characteristic of a positively curved Riemannian manifold that admits a large isometric torus action. We apply our results to prove obstructions to symmetric spaces, products of manifolds, and connected sums admitting positively curved metrics with symmetry.
\end{abstract}

\maketitle \thispagestyle{empty}

\section*{Introduction}
\bigskip

The question of whether a given smooth manifold admits a Riemannian metric with positive sectional curvature is nearly as old as the subject of Riemannian geometry itself. However, there are not many known examples of manifolds admitting positive sectional curvature. In addition to spheres and projectives spaces, only a short list of families arise (see Ziller \cite{Ziller07} for a survey of examples, and see Dearricott \cite{Dearricott11}, Grove--Verdiani--Ziller \cite{GroveVerdianiZiller11}, and Petersen--Wilhelm \cite{PetersenWilhelm09pre} for two new examples in dimension seven).

The complementary task of proving obstructions to positive curvature has proven to be just as difficult. In fact, for simply connected closed manifolds, every known topological obstruction to positive sectional curvature is already an obstruction to non-negative sectional curvature.

In the 1990s, Karsten Grove initiated a research program that has been breathing new life into these problems for two decades. The idea is to focus on positively curved metrics that admit large isometry groups. The measure of symmetry we will consider in this paper is the \textit{symmetry rank}, which is the rank of the isometry group. We recall that the symmetry rank of a Riemannian manifold is at least $r$ if and only if there exists an effective, isometric action a torus of dimension $r$.

Many topological classification results of varying strengths (diffeomorphism, homeomorphism, homotopy, etc.) have been proven under the assumption that the symmetry rank is sufficiently large. The classification theorems of Grove--Searle (Theorem \ref{thm:GroveSearle}) and Wilking (Theorem \ref{thm:Wilking}) are prototypical examples. For related results and context, we refer the reader to the comprehensive surveys of Grove \cite{Grove09} and Wilking \cite{Wilking07}.

Our first result concerns the Euler characteristic of positively curved manifolds with symmetry. The topological classifications of Grove--Searle and Wilking imply complete calculations of the Euler characteristic under linear symmetry rank assumptions. In \cite{Kennard1}, the second author proved that $\chi(M) > 0$ for all positively curved $(4k)$--dimensional manifolds $M$ with symmetry rank at least $2\log_2(4k)-2$. Under a slightly larger logarithmic symmetry assumption, we prove the following upper bound for $\chi(M)$. (See Theorem \ref{LogAssumptionPLUS} for a stronger, but more involved, statement.)

\begin{main}\label{LogAssumption}
If $M^n$ is a closed Riemannian manifold with positive sectional curvature and symmetry rank at least $\log_{4/3}(n)$, then
	$\chi(M) < 2^{3(\log_2 n)^2}$.
\end{main}

We would like to highlight connections to some conjectured upper bounds for the Betti numbers of non-negatively curved manifolds. Since $\chi(M) \leq \sum b_i(M)$, such bounds immediately imply bounds for the Euler characteristic.

Gromov's Betti number estimate provides one such bound. Specifically, Gromov proved that there exists a constant $C(n)$ such that $\sum b_i(M) \leq C(n)$ for all $n$--manifolds $M$ that admit a non-negatively curved metric (see \cite{Gromov81}). It is conjectured that $C(n)$ can be replaced by $2^n$, the sum of the Betti numbers of the $n$--dimensional torus.

The second conjectured upper bound follows from a combination of the Bott--Grove--Halperin conjecture and the Hopf conjecture. Bott--Grove--Halperin conjectured that manifolds admitting non-negative curvature are rationally elliptic (see Grove \cite[Section 5]{Grove02}), and Hopf conjectured that even-dimensional manifolds admitting positive curvature have positive Euler characteristic. Given these conjectures, it follows that an oriented, closed, even-dimensional manifold $M^n$ that admits positive curvature has Euler characteristic at most $2^{n/2}$. Since the estimate in Theorem \ref{LogAssumption} is asymptotically smaller than both $2^n$ and $2^{n/2}$, even a verification of these conjectures
would not imply our result.

\smallskip

Our second result provides a stronger bound on $\chi(M)$ by assuming an asymptotically larger symmetry assumption than in Theorem \ref{LogAssumption}. Specifically, the symmetry assumption is linear in the dimension, however the slope is allowed to be arbitrarily small.

\begin{main}\label{LinearAssumptionDelta}
Let $\delta>0$. There exists a constant $c = c(\delta)$ such that the following holds: If a closed Riemannian manifold $M^n$ has positive sectional curvature and an effective, isometric action by a torus $T$ with $\dim(T) \geq \delta n$, then the fixed-point set $M^T$ has at most $cn$ components and $\chi(M) < c n \log_2 n$.

\end{main}

See Theorem \ref{LinearAssumptionAlpha} for a more elaborate statement containing the concretely stated constant $c(\delta)$.
To be specific in just one case, we mention the following consequence of Theorem \ref{LinearAssumptionAlpha}: If $M^{2n}$ is a closed, positively curved Riemannian manifold, and if $M$ admits an effective, isometric action by $T^r$ with
\begin{align*}	
r \geq \frac{(2n)}{8} + 2\log_2(2n) + 1,
\end{align*}
then the fixed-point set of the torus action has fewer than $\frac{3}{2}\chi(\cc\pp^n)$ components and
\begin{align*}
\chi(M) < \chi(\cc\pp^n)(1 + \log_2 n).
\end{align*}

\smallskip

Let us state two corollaries of Theorem \ref{LogAssumption}. They relate to Hopf's conjecture that no positively curved metric exists on $\s^2\times\s^2$. More generally, it is conjectured that no positively curved metric exists on a non-trivial product of simply connected closed manifolds or on a compact symmetric space of rank greater than one.
\begin{maincor}[Stable Hopf conjecture with symmetry]\label{cor:StableHopf}
Let $M^n$ be a closed, one-connected Riemannian manifold with even dimension, positive sectional curvature, and symmetry rank at least $\log_{4/3} n$.
\begin{itemize}
\item If $M=N^{\times k} = N \times N \times \cdots \times N$, then
$k < 3(\log_2 n)^2$.
\item If $M=N^{\# k} = N \# N \# \cdots \# N$ with $\chi(N) \neq 2$, then $k < 2^{3(\log_2 n)^2}$.
\end{itemize}
\end{maincor}

We remark on the name of the corollary: In the first conclusion, we may start with a manifold $N^m$ and consider positively curved metrics on $N^{\times k}$ with symmetry rank $\log_{4/3}(\dim(N^{\times k}))$. The conclusion in this case is that $k < 3(\log_2(km))^2$, which again cannot be true for arbitrarily large $k$.

We also remark that, in some cases, we can replace $N^{\times k}$ by more general products or by total spaces of iterated fibrations. See Section \ref{Corollaries} for related statements.

\begin{maincor}\label{cor:SymmetricSpaces}
Assume $M^{2n}$ is a positively curved Riemannian manifold with the rational homotopy type of a simply connected, compact symmetric space $N$. If $M$ has symmetry rank at least $\log_{4/3} n+ 7$, then $N$ has the form $Q \times \s^{2n_1} \times\cdots\times \s^{2n_s}$ for some $0\leq s < 3(\log_2 n)^2$ and some
	\[Q\in\{\s^{2q}, \cc \pp^q, \hh \pp^q\}\cup\{\SO(p+q)/\SO(p)\times\SO(q)\st 2\leq p\leq3\}.\]
\end{maincor}

This improves Theorem A in \cite{Kennard2} by bounding the number of spherical factors. The proofs of Corollaries \ref{cor:StableHopf} and \ref{cor:SymmetricSpaces} are contained in Section \ref{Corollaries}.

\smallskip

Our next result concerns the positivity of the Euler characteristic. In particular, we require it to deduce Corollaries \ref{cor:StableHopf} and \ref{cor:SymmetricSpaces}.
\begin{main}\label{thm:OddBettis}
Let $M^{2n}$ be a one-connected, closed Riemannian manifold with positive sectional curvature and symmetry rank at least $2\log_2(2n) + 2$. If the second, third, or fourth Betti number of $M$ vanishes, then $\chi(M) \geq 2$.
\end{main}
The proof follows from a collection of previous results. It does however extend Theorem A of \cite{Kennard1}, at least for manifolds with symmetry rank $r \geq 2\log_2 n + 2$. For a discussion of related results, see Section \ref{sec:OddBettis}.

\smallskip

We take a moment here to summarize the techniques involved in the proofs, as well as to describe the layout of this article. Perhaps most importantly, Wilking's connectedness lemma and the resulting induction machinery and cohomological periodicity play a large role. We summarize these and other required results in Section \ref{Preliminaries}. Building upon this, we prove three new, convenient sufficient conditions for a manifold to have 4--periodic rational cohomology in Section \ref{NewTools} (see Propositions \ref{pro:Case1lemma}, \ref{pro:dk2}, and \ref{pro:new4per}).

Another key aspect of the proofs involves proving bounds on the topology of the fixed-point set $M^T$ of the torus action. For Theorem \ref{LogAssumption}, we accomplish this by adapting ideas from the first named author's Ph.D. thesis \cite{Amann09thesis}, where positive quaternionic K\"ahler manifolds with symmetry are studied. For Theorem \ref{LinearAssumptionDelta}, we analyze the direct sum of isotropy representations over multiple fixed points. The latter idea is a useful observation that does not seem to have been made before. It was inspired by a conversation with Burkhard Wilking (see Proposition \ref{pro:Griesmer}, and the discussion in Section \ref{NewTools}).

Sections \ref{ProofOfLogAssumption} and \ref{ProofOfLinearAssumptionDelta} contain the proofs of Theorems \ref{LogAssumption} and \ref{LinearAssumptionDelta}, respectively. We wish to emphasize that the proofs are independent of each other. In particular, among the four propositions we prove in Section \ref{NewTools}, only Propositions \ref{pro:Case1lemma} and \ref{pro:dk2} are used for Theorem \ref{LogAssumption}, while Theorem \ref{LinearAssumptionDelta} requires only Propositions \ref{pro:dk2}, \ref{pro:new4per}, and \ref{pro:Griesmer}.

Finally, the proof of Theorem \ref{thm:OddBettis} only requires results in Section \ref{Preliminaries}, as well as a few additional arguments. It is contained in Section \ref{sec:OddBettis}. Corollaries \ref{cor:StableHopf} and \ref{cor:SymmetricSpaces} follow from Theorems \ref{LogAssumption}, \ref{LinearAssumptionDelta}, and \ref{thm:OddBettis}. The proofs are contained in Section \ref{Corollaries}.

\bigskip

\bigskip

\ack We are grateful to Burkhard Wilking for a motivating discussion, as well as to Anand Dessai, Nicolas Weisskopf, and Wolfgang Ziller for commenting on earlier versions. We also wish to thank the referees for several suggestions to improve the presentation of the article.


\bigskip
\section{Preliminaries}\label{Preliminaries}
\bigskip

In this section, we summarize a number of previous results that will be used in the proofs in this article. We begin with two important results referenced in the introduction.

\begin{theo}[Maximal symmetry rank classification, Grove--Searle \cite{GroveSearle94}]\label{thm:GroveSearle}
If $T^r$ acts effectively by isometries on a closed, one-connected, positively curved Riemannian manifold, then $r \leq \floor{\frac{n+1}{2}}$ with equality only if $M$ is diffeomorphic to $\s^n$ or $\cc\pp^{n/2}$.
\end{theo}

\begin{theo}[Wilking, \cite{Wilking03}]\label{thm:Wilking}
Let $M^n$ be a closed, one-connected Riemannian manifold with positive sectional curvature and symmetry rank $r$.
\begin{itemize}
\item (Homotopy classification) If $r\geq \frac{n}{4} + 1$ and $n \geq 10$, then $M$ is homotopy equivalent to $\s^n$, $\cc \pp^{n/2}$, $\hh\pp^{n/4}$, or $\mathrm{Ca}\pp^2$.
\item (Cohomology classification) If $r\geq \frac{n}{6}+1$ and $n\geq 6000$, then
	\begin{enumerate}
	\item $M$ is homotopy equivalent to $\s^n$ or $\cc\pp^{n/2}$,
	\item $M$ has the integral cohomology of $\hh\pp^{n/4}$,
	\item $n\equiv 2\bmod{4}$ and, for all fields $\mathbb{F}$, the cohomology algebra of $M$ with coefficients in $\mathbb{F}$ is isomorphic to either $H^*(\cc\pp^{n/2};\mathbb{F})$ or $H^*(\s^2\times\hh\pp^{\frac{n-2}{4}};\mathbb{F})$, or
	\item $n\equiv 3\bmod{4}$ and, for all fields $\mathbb{F}$, the cohomology algebra of $M$ with coefficients in $\mathbb{F}$ is isomorphic to either $H^*(\s^n;\mathbb{F})$ or $H^*(\s^3\times\hh\pp^{\frac{n-3}{4}})$.
	\end{enumerate}
\end{itemize}
\end{theo}

Next, we state a number of results whose importance to the Grove research program is already well established.
\begin{theo}[Conner, \cite{Conner57}]\label{thm:Conner}
If $T$ acts smoothly on a closed manifold $M$, then $\chi(M^T) = \chi(M)$. Moreover,
	\begin{enumerate}
	\item $\sum b_{2i}(M^T) \leq \sum b_{2i}(M)$, and
	\item $\sum b_{2i+1}(M^T) \leq \sum b_{2i+1}(M)$.
	\end{enumerate}
\end{theo}
The first statement will allow us to pull some information about $M^T$ up to $M$, while the second and third results will allow us to push information about $M$ or its submanifolds down to $M^T$.

\begin{theo}[Berger, \cite{Berger66,GroveSearle94}]\label{Berger}
Suppose $T$ is a torus acting by isometries on a closed, positively curved manifold $M^n$. If $n$ is even, then the fixed-point set $M^T$ is nonempty, and if $n$ is odd, then a codimension one subtorus has nonempty fixed-point set.
\end{theo}

\begin{theo}[Frankel, \cite{Frankel61}]
Let $M^n$ be a closed, compact Riemannian manifold with positive sectional curvature. If $N_1^{n_1}$ and $N_2^{n_2}$ are totally geodesic submanifolds of $M$ with $n_1 + n_2 \geq n$, then $N_1\cap N_2$ is nonempty.
\end{theo}

The bound on the dimensions of $N_1$ and $N_2$ is equivalent to the bound $n - k_1 - k_2 \geq 0$, where $k_i$ is the codimension of $N_i$. This result was greatly improved by Wilking \cite{Wilking03}:
\begin{theo}[Connectedness lemma]\label{ConnectednessTheorem}
Suppose $M^n$ is a closed Riemannian manifold with positive sectional curvature.
 \begin{enumerate}
  \item If $N^{n-k}$ is a closed, embedded, totally geodesic submanifold of $M$, then $N\embedded M$ is $(n- 2k + 1)$--connected.
  \item If $N_1^{n-k_1}$ and $N_2^{n-k_2}$ are closed, embedded, totally geodesic submanifolds of $M$ with $k_1\leq k_2$, then $N_1\cap N_2\embedded N_2$ is $(n - k_1 - k_2)$--connected.
 \end{enumerate}
\end{theo}
Recall that an inclusion $N\to M$ is called $h$--connected if $\pi_i(M,N)=0$ for $i\leq h$. As a consequence, the map $H_i(N;\Z) \to H_i(M;\Z)$ induced by inclusion is an isomorphism for $i<h$ and a surjection for $i = h$, and the map $H^i(M;\Z) \to H^i(N;\Z)$ induced by inclusion is an isomorphisms for $i<h$ and an injection for $i=h$.

Given a highly connected inclusion of closed, orientable manifolds, applying Poincar\'e duality to each manifold produces a certain periodicity in cohomology (again see \cite{Wilking03}):

\begin{theo}\label{WilkingPD}
Let $M^n$ and $N^{n-k}$ be closed, connected, oriented manifolds. If the inclusion $N\embedded M$ is $(n-k-l)$--connected, then there exists $x\in H^k(M;\Z)$ such that the map $H^i(M;\Z) \to H^{i+k}(M;\Z)$ induced by multiplication by $x$ is a surjection for $l\leq i < n-k-l$ and an injection for $l<i\leq n-k-l$.
\end{theo}

This periodicity is especially strong when $l = 0$. For example, because $M$ is connected, $x^i$ generates $H^{ik}(M;\Z)$ for $0 < ik < n$. We call this property periodicity (see \cite{Kennard1}):

\begin{defi}\label{def:periodic} For an integer $k > 0$, a ring $R$, and a connected space $M$, we say that $H^*(M;R)$ is \textit{$k$--periodic} if there exists $x\in H^k(M;R)$ such that the map $H^i(M;R) \to H^{i+k}(M;R)$ induced by multiplication by $x$ is surjective for $0\leq i< n - k$ and injective for $0 < i \leq n - k$.
\end{defi}

Examples of manifolds with periodic cohomology include spheres, projective spaces, and some products such as $\s^2\times\hh \pp^m$ or $N\times \mathrm{Ca}\pp^2$ for any manifold $N$ of dimension less than eight.

In practice, we conclude periodicity from the connectedness lemma and Theorem \ref{WilkingPD}. Specifically, if $N_1$ and $N_2$ intersect transversely as in part (2) of the connectedness lemma, then $H^*(N_2;\Z)$ is $k_1$--periodic. It follows that the rational cohomology of $N_2$ is also $k_1$--periodic. The following refinement of this rational periodicity is proved in \cite{Kennard1}:

\begin{theo}[Periodicity theorem]\label{per}
Let $M^n$ be a closed, one-connected Riemannian manifold with positive sectional curvature. Let $N_1^{n-k_1}$ and $N_2^{n-k_2}$ be connected, closed, embedded, totally geodesic submanifolds that intersect transversely.
	If $2k_1 + 2k_2 \leq n$, the rational cohomology rings of ${M}$, ${N_1}$, ${N_2}$, and ${N_1\cap N_2}$ are $4$--periodic. In particular, $b_{2i}(M) \leq 1$ for all $i$.
\end{theo}

While these connectedness and periodicity theorems have no symmetry assumptions, their main applications have been to the study of positively curved manifolds with large symmetry. Indeed, fixed point sets of isometries are totally geodesic, and with enough isometries, one can guarantee the existence of totally geodesic submanifolds with small codimension and pairs of submanifolds that intersect transversely. In the next section, we illustrate how this works by considering three situations in which one can use the periodicity theorem to conclude that a manifold has $4$--periodic rational cohomology.

Another consequence of the connectedness and periodicity theorems is the following, which we use in the proofs of Proposition \ref{pro:new4per} and Theorem \ref{thm:OddBettis}.

\begin{theo}[\cite{Kennard2}, Theorem 2.2]\label{thmP}
Let $n\geq c\geq 2$ be even integers, and let $M^n$ be a closed, one-connected manifold with positive sectional curvature. Assume $T$ is a torus acting effectively by isometries on $M$ such that
	\[\dim(T) \geq 2\log_2 n + \frac{c}{2} - 1.\]
For all $x\in M^T$, there exists $H\subseteq T$ such that $N=M^H_x$ satisfies
	\begin{enumerate}
	\item $H^*(N;\Q)$ is $4$--periodic,
	\item $N \subseteq M$ is $c$--connected,
	\item $\dim(N) \equiv n\bmod{4}$, and
	\item $\dim(N) \geq c+4$.
	\end{enumerate}
In particular, if $\dim(N)\geq 8$, there exists $x\in H^4(M;\Q)$ such that the maps $H^i(M;\Q) \to H^{i+4}(M;\Q)$ induced by multiplication by $x$ are surjective for $0 \leq i \leq c - 4$ and injective for $0 < i \leq c - 4$.
\end{theo}

\bigskip
\section{New tools}\label{NewTools}
\bigskip

In this section, we prove four propositions that provide the essential tools for the proofs of the main theorems. The first three (Propositions \ref{pro:Case1lemma}, \ref{pro:dk2}, and \ref{pro:new4per}) use the connectedness and periodicity theorems. They provide convenient sufficient conditions for a manifold to have $4$--periodic rational cohomology. Since such manifolds have Euler characteristic and even Betti numbers bounded above by those of complex projective space, these results provide important information required for the proofs of Theorems \ref{LogAssumption} and \ref{LinearAssumptionDelta}.

The final proposition (Proposition \ref{pro:Griesmer}) builds on a technique of Wilking whereby one applies the theory of error correcting codes to isotropy representations. Inspired  by a suggestion of Wilking, we extend this idea to the direct sum of isotropy representations. The result is that we obtain components of fixed-point sets of involutions that contain a large portion of the fixed-point set of the torus action. Proposition \ref{pro:Griesmer} is crucial to the proof of Theorem \ref{LinearAssumptionDelta}.

\begin{prop}\label{pro:Case1lemma}
Let $M^n$ be a closed, one-connected Riemannian manifold with positive sectional curvature, and let $T$ be a torus acting effectively by isometries on $M$. Assume $M^T$ is nonempty, and let $x\in M^T$. If there exists a subgroup $\Z_2^j \subseteq T$ with $j\geq \floor{\log_2 n} - 1$ such that every $\tau\in\Z_2^j$ satisfies $\cod(M^\tau_x) \leq n/4$, then $H^*(M;\Q)$ is $4$--periodic.
\end{prop}

\begin{proof}
Inductively choose a $j$ algebraically independent involutions \linebreak[4]$\tau_1,\ldots,\tau_j\in\Z_2^j$ such that
	\[k_i = \cod\of{(N_{i-1})^{\tau_i}_x \subseteq N_{i-1}} \geq
	        \cod\of{(N_{i-1})^{\tau}_x   \subseteq N_{i-1}}\]
for all $\tau\in\Z_2^i\setminus\langle\tau_1,\ldots,\tau_{i-1}\rangle$, where $N_h = M^{\langle\tau_1,\ldots,\tau_h\rangle}_x$ for all $h\geq 1$ and $N_0=M$.

By maximality of $k_i$, we have both of the estimates
	\begin{eqnarray*}
	k_i &\geq& \cod\of{(N_{i-1})^{\tau_{i+1}}_x \subseteq N_{i-1}} = k_{i+1} + l~\mathrm{and}\\
	k_i &\geq& \cod\of{(N_{i-1})^{\tau_i\tau_{i+1}}_x \subseteq N_{i-1}} = k_{i+1} + (k_i-l),
	\end{eqnarray*}
where $l$ is the dimension of the intersection of the $(-1)$--eigenspaces of the actions of $\tau_{i+1}$ and $\tau_i\tau_{i+1}$ on the normal space of $N_{i-1}$. Geometrically, $l$ is the codimension of $(N_{i-1})^{\langle\tau_i,\tau_{i+1}\rangle}_x$ inside $(N_{i-1})^{\tau_i\tau_{i+1}}_x$. Together, these two estimates imply that $k_i\geq 2k_{i+1}$ for all $i\geq 1$. We draw two conclusions from this. First,
	\[k_j \leq k_{j-1}/2 \leq \cdots \leq k_1/2^{j-1} \leq n/2^{j+1} \leq n/2^{\floor{\log_2 n}} < 2.\]
Since $k_j$ is even, we conclude that $k_j = 0$. Second, for all $1\leq h\leq j$,
	\[4k_h + (k_{h-1} + \ldots + k_1) \leq \frac{n}{2} + \frac{n}{2^h} \leq n,\]
so we can conclude both of the estimates $2k_h \leq \dim\of{N_{h-1}}$ and
	\[\dim\of{N_{h-1}} - 2\cod\of{N_h \subseteq N_{h-1}} + 1 > \frac{1}{2} \dim\of{N_{h-1}}.\]
We will use the first estimate in a moment. The second estimate implies that the inclusion $N_h \to N_{h-1}$ is $c$--connected with $c > \frac{1}{2}\dim{N_{h-1}}$. As a consequence, if $H^*(N_h;\Q)$ is $4$--periodic for some $h$, then $H^*(M;\Q)$ would be $4$--periodic as well. We will use this fact later.

Let $2\leq i \leq j$ be the the minimal index such that $k_i = 0$. For $1\leq h < i$, let $l_h$ denote the dimension of the intersection of the $(-1)$--eigenspaces of the action of $\tau_h$ and $\tau_i$ on $T_x(N_{h-1})$. Geometrically $l_h$ is the codimension of $(N_{h-1})^{\langle\tau_h,\tau_i\rangle}_x$ inside $(N_{h-1})^{\tau_h\tau_i}_x$. By replacing $\tau_i$ by $\tau\tau_i$ for some $\tau\in\langle\tau_1,\ldots,\tau_{i-1}\rangle$, we may assume that $l_h \leq k_h/2$ for all $1\leq h<i$.

Let $1\leq h < i$ be the maximal index such that $l_h > 0$. Observe that some $l_h>0$ since the subgroup $\Z_2^j$ acts effectively on $M$. It follows that $N_h$ is the transverse intersection of $(N_{h-1})^{\tau_i}_x$ and $(N_{h-1})^{\tau_i\tau_h}$ inside $N_{h-1}$. Since
	\[2\cod\of{(N_{h-1})^{\tau_i}_x       \subseteq N_{h-1}}
	 +2\cod\of{(N_{h-1})^{\tau_i\tau_h}_x \subseteq N_{h-1}}\]
is equal to $2l_h + 2(k_h-l_h) = 2k_h \leq \dim\of{N_{h-1}}$, we conclude that \linebreak[4]$H^*(N_{h-1};\Q)$ is $4$--periodic by the periodicity theorem. As already established, this implies that $H^*(M;\Q)$ is $4$--periodic.
\end{proof}

The second sufficient condition for rational $4$--periodicity is an immediate consequence of the periodicity theorem.

\begin{prop}\label{pro:dk2}
Let $M^n$ be a closed, one-connected, positively curved Riemannian manifold, let $T$ be a torus acting isometrically and effectively on $M$, and let $T' \subseteq T$ denote a subtorus that fixes a point $x \in M$.

If there exists an involution $\sigma \in T'$ such that $M^\sigma_x$ has $\dim(M^\sigma_x) \geq \frac{n}{2}$ and is fixed by another involution in $T'$, then $M$ has $4$--periodic rational cohomology. In particular, this applies if $M^\sigma_x$ is fixed by a two-dimensional torus in $T'$.
\end{prop}

\begin{proof}
Choose a non-trivial involution $\tau\in T'$ not equal to $\sigma$. Since $M^\sigma_x$ is fixed by $\tau$, it is contained in $M^\tau_x$. Moreover, both inclusions $M^\sigma_x \subseteq M^\tau_x \subseteq M$ are strict since the action of $T$ is effective.

Similarly, $M^\sigma_x \subseteq M^{\sigma\tau}_x \subseteq M$ with both inclusions strict. From the isotropy representation of $\Z_2^2=\langle\sigma,\tau\rangle\subseteq T'$ at the fixed point $x$, it follows that $M^\sigma_x$ is the transverse intersection of $M^\tau_x$ and $M^{\sigma\tau}_x$. Since
	\[2\cod(M^\tau_x) + 2\cod(M^{\sigma\tau}_x) = 2\cod(M^\sigma_x) \leq n,\]
the periodicity theorem implies that $H^*(M;\Q)$ is $4$--periodic.
\end{proof}

The third sufficient condition for rational $4$--periodicity is the following. Its proof is a simple combination of the connectedness lemma and Theorems \ref{WilkingPD} and \ref{thmP}.

\begin{prop}\label{pro:new4per}
Let $M^n$ be a closed, one-connected manifold with positive sectional curvature. If $M$ both contains a totally geodesic submanifold of codimension $k\leq \frac{n+2}{4}$ and has symmetry rank at least
$2\log_2 n + \frac{k}{2} + 1$, then $H^*(M;\Q)$ is $4$--periodic.
\end{prop}

\begin{proof}
By \thm{WilkingPD}, the existence of a codimension $k$ totally geodesic submanifold implies that there exists $y \in H^k(M;\Q)$ such that the maps $H^i(M;\Q) \to H^{i+k}(M;\Q)$ given by multiplication by $y$ are injective for $k - 1 < i \leq n-2k + 1$ and surjective for $k-1 \leq i < n- 2k + 1$.

Next by \thm{thmP} with $c = k+4$, the symmetry rank assumption implies the existence of $x\in H^4(M;\Q)$ such that the maps $H^i(M;\Q) \to H^{i+4}(M;\Q)$ are surjective for $0\leq i \leq k$ and injective for $0 < i \leq k$. In particular, $y = ax^{\frac{k-m}{4}}$ for some $0 \leq m < 4$ and $a\in H^m(M;\Q)$. We claim that $x$ induces periodicity in $H^*(M;\Q)$.

We first show that it suffices to show injectivity of the multiplication maps $H^i(M;\Q) \to H^{i+4}(M;\Q)$ for $0<i\leq n-4$. Indeed, Poincar\'e duality would imply that the Betti numbers satisfy
	\[b_i \leq b_{i + 4} = b_{n - i - 4} \leq b_{n-i} = b_i\]
for all $0<i<n-4$, hence the injective maps $H^i(M;\Q) \to H^{i+4}(M;\Q)$ must be isomorphisms for dimensional reasons. Since the map $H^0(M;\Q)\to H^4(M;\Q)$ is already known to be surjective by our choice of $x$, the proof would be complete.

We now prove that multiplication maps $H^i(M;\Q)\to H^{i+4}(M;\Q)$ are injective for $0 < i \leq n-4$. Since $k\leq\frac{n+2}{4}$, the four cases that follow constitute a proof:

\begin{itemize}
	\item If $0 < i \leq k$, injectivity follows by our choice of $x$.
	\item If $k \leq i \leq n-2k+1$, injectivity follows since $x$ is a factor of $y$ and since $H^i(M;\Q) \stackrel{\cdot y}{\longrightarrow} H^{i+k}(M;\Q)$ is injective for these values of $i$.
	\item If $2k \leq i \leq n - k - 3$, injectivity follows from the fact that
		\[H^{i-k}(M;\Q) \stackrel{\cdot x}{\longrightarrow} H^{i-k+4}(M;\Q) \stackrel{\cdot y}{\longrightarrow} H^{i+4}(M;\Q)\]
is a composition of injective maps that equals the composition
		\[H^{i-k}(M;\Q) \stackrel{\cdot y}{\longrightarrow} H^{i}(M;\Q) \stackrel{\cdot x}{\longrightarrow} H^{i+4}(M;\Q),\]
	the first map of which is an isomorphism.
	\item If $n - k - 4 < i \leq n - 4$, injectivity follows from a direct argument. Let $z\in H^i(M;\Q)$ be a nonzero element. By Poincar\'e duality, there exists $w\in H^{n-i}$ such that $zw\neq 0$. The multiplicative property of $x$ implies that $w = xw'$ for some $w'\in H^{n-i-4}$. Hence
		\[0\neq zw = z(xw') = (zx)w',\]
which implies $xz\neq 0$, completing the proof.
\end{itemize}
\end{proof}

The final tool we require in the proofs involves an application of the theory of error-correcting codes. This idea was used for the first time in the subject in Wilking \cite{Wilking03} to produce fixed-point sets of isometries of torus actions that have small codimension. In the setting of positively curved manifolds, the inclusion maps of these totally geodesic fixed-point sets are highly connected by the connectedness lemma, hence one can construct inductive arguments over dimension as in \cite{Wilking03}. In addition, one obtains periodicity results as discussed in Section \ref{Preliminaries}.

We briefly summarize how this technique works. First, one requires a fixed point $x\in M^T$ of a torus action on $M$. By restricting attention to the subgroup of involutions in $T$, one associates to the isotropy represention a linear error-correcting code. One can then apply algebraic bounds from the theory of error-correcting codes to make conclusions about the Hamming weights of the code. Given this, one translates this algebraic information into geometric data about the fixed-point sets of involutions in the torus.

Here, inspired by a suggestion of Wilking, we consider an extension of this idea. The starting point is to simultaneously consider the isotropy representations at multiple fixed points $x_i \in M^T$. Specifically, one observes that the direct sum of isotropy representions, when restricted to the subgroup of involutions, yields another error-correcting code. Hence one can implement the strategy outlined above. The translation of algebraic information into geometric data is less immediate. However, if $M$ is positively curved, one can apply Frankel's theorem to prove that the chosen components of fixed-point sets of involutions have, in fact, a large number of components of $M^T$ contained in them.
The specific result we will use is the following.
\begin{prop}\label{pro:Griesmer}
Let $n\geq c\geq 0$ be even. Suppose a torus $T^r$ acts effectively and isometrically on a closed, connected, positively curved manifold $M^n$. If $x_1,\ldots,x_t\in M$ are fixed points of the $T^r$--action on $M$, and if
\[r \geq \frac{tc}{2} + \floor{\log_2(tn - tc + 2)},\] then there exists a nontrivial involution $\iota\in T^r$
and a component $N\subseteq M^\iota$ with $\cod(N)\leq\frac{n-c}{2}$ such that $N$ contains $\ceil{\frac{t+1}{2}}$ of the $x_i$.
\end{prop}

The $t=1$ statement is essentially Lemma 1.8 in \cite{Kennard2}, and it implies that we can cover each fixed-point component of $M^T$ by a component of a fixed-point set of an involution with dimension at least $n/2$. An argument using the fact $T^r$ has only $2^r$ involutions, together with Frankel's theorem, will be the key step in the proof of Theorem \ref{LogAssumption}.

The conclusion for general $t$ is really only used to prove Theorem \ref{LinearAssumptionDelta}, but for this it is crucial. The basic idea for the proof of the first conclusion of Theorem \ref{LinearAssumptionDelta} is that, by iterated use of this proposition, one can find involutions $\iota_1,\ldots,\iota_h\in T^r$ and components $N_i\subseteq M^{\iota_i}$ such that
	\begin{itemize}
	\item $h = \floor{\log_2(t+1)}$,
	\item $\cod(N_i)\leq\frac{n-c}{2}$ for all $i$, and
	\item $N_1\cup N_2\cup\cdots\cup N_h$ contains $x_i$ for all $1 \leq i \leq t$.
	\end{itemize}
A variation of this idea is used again in the proof of the second conclusion of Theorem \ref{LinearAssumptionDelta}.

\begin{proof}
Choose a basis for each $T_{x_i} M$ so that the isotropy representation $T^r \to \SO(T_{x_i} M)$ factors through $\U(n/2) \subseteq \SO(T_{x_i} M)$. It follows that the $\Z_2^r \subseteq T^r$ maps via the isotropy representation to a copy of $\Z_2^{n/2}\subseteq \SO(T_{x_i} M)$. Moreover, the map $\Z_2^r \to \Z_2^{n/2}$ has the property that the Hamming weight of the image of $\iota\in\Z_2^r$ is half of the codimension of $M^\iota_{x_i}$.

Consider the direct sum $T^r \to \SO\of{\bigoplus_{i=1}^t T_{x_i} M}$ of the isotropy representations at $x_1,\ldots,x_t$. Restricting to the involutions in $T^r$, we obtain a homomorphism
	\[\Z_2^r \to \bigoplus_{i=1}^t \Z_2^{n/2} \cong \Z_2^{tn/2}.\]
Let $\pi_i$ denote the projection of the direct sum onto the $i$--th summand of $\Z_2^{n/2}$. The Hamming weight of the image of $\iota\in\Z_2^r$ under the composition of the isotropy representation and $\pi_i$ is half the codimension $k_i$ of $M^\iota_{x_i}$. Set $k = \sum k_i$.

We claim that a nontrivial $\iota\in \Z^r$ exists with $k\leq\frac{t(n-c)}{2}$. Assuming this for a moment, we conclude the proof. Relabel the $x_i$ so that the codimensions $k_i = \cod(M^\iota_{x_i})$ are increasing in $i$. It follows that, for all $1 < j \leq \ceil{\frac{t+1}{2}}$,
	\[\pfrac{t}{2}(k_1 + k_j) \leq \sum_{i=1}^t k_i \leq \frac{t(n-c)}{2} \leq \frac{tn}{2}.\]
Using Frankel's theorem, we conclude that the components $M^\iota_{x_1}$ and $M^\iota_{x_j}$ intersect and hence coincide. Since we also have
	\[tk_1 \leq \sum_{i=1}^t k_i \leq \frac{t(n-c)}{2},\]
we conclude the component $M^\iota_{x_1}$ of $M^\iota$ satisfies $\cod(M^\iota_{x_1})\leq\frac{n-c}{2}$ and contains at least $\ceil{\frac{t+1}{2}}$ of the $x_i$. This concludes the proof given the existence of $\iota\in \Z^r$ as above.

We proceed by contradiction to show that some nontrivial $\iota\in \Z^r$ exists with $k\leq\frac{t(n-c)}{2}$. Suppose therefore that $k \geq \frac{t(n-c) + 2}{2}$ for all nontrivial $\iota\in\Z_2^r$. Applying Griesmer's bound, we obtain
	\[\frac{tn}{2} \geq \sum_{i=0}^{r-1} \ceil{\frac{t(n-c)+2}{2^{i+2}}}.\]
We estimate the sum on the right-hand side as follows:
	\[\frac{tn}{2} 	\geq \sum_{i=0}^{r-\frac{tc}{2}-1} \frac{tn-tc+2}{2^{i+2}} + \sum_{i=r-\frac{tc}{2}}^{r-1} 1
				= \frac{tn-tc+2}{2} - \frac{tn-tc+2}{2^{r - tc/2+1}} + \frac{tc}{2}.\]
Rearranging, we conclude that
	\[r \leq \frac{tc}{2} -1 + \floor{\log_2(tn-tc+2)}.\]
But this contradicts the assumed bound on $r$, so the result follows.
\end{proof}

\smallskip
\section{Proof of Theorem \ref{LogAssumption}}\label{ProofOfLogAssumption}
\bigskip

We first state Theorem \ref{LogAssumptionPLUS}, which is the most general version of Theorem \ref{LogAssumption} we can prove. It provides an upper bound on the Euler characteristic of the form $\chi(M) \leq f_0(n)$. We then deduce Theorem \ref{LogAssumption} from Theorem \ref{LogAssumptionPLUS} by proving the estimate $f_0(n) < 2^{3(\log_2 n)^2}$. The heart of this section is the proof of Theorem \ref{LogAssumptionPLUS}. At the end of the section, we comment on the quality of the approximation provided by Theorem \ref{LogAssumption} and illustrate the result graphically.

In order to state Theorem \ref{LogAssumptionPLUS}, we require two definitions. Set
	\[s(n) = \floor{\log_2 n} + \floor{\log_2(n+2)} - 2,\]
for $n>0$, set $f_0(n) = \frac{n}{2} + 1$ for $n\leq 52$, and define $f_0(n)$ for $n\geq 54$ by the recursive formula
		\[f_0(n) = \of{2^{s(n)} - 1} f_0\of{2\floor{\frac{3n-4}{8}}}.\]

\begin{theo}\label{LogAssumptionPLUS}
Let $M^n$ be a connected, closed, even-dimensional Riemannian manifold with positive sectional curvature. If a torus $T$ of dimension at least $\log_{4/3}(n)$ acts effectively and isometrically on $M$, then \[\chi(M) \leq \sum b_{2i}(M^T) \leq f_0(n).\]
\end{theo}

We briefly provide some explanation for the definition of $f_0(n)$. We suggest to see its definition as being motivated by two facts, the first providing the $2^{s(n)} - 1$ factor, and the second yielding the argument $2\floor{\frac{3n-4}{8}}\approx \frac{3}{4}n$. The latter corresponds to the base of $4/3$ of the logarithms we use.

More specifically, in the proof of Case 2 below, we will prove that $M^T$ is contained in the union of submanifolds $N_j$. It follows that $\sum b_{2i}(M^T)$ is bounded above by $\sum_j \sum_i b_{2i}(N_j^T)$. Now the $N_j$ correspond to unique, non-trivial involutions in a subgroup $\Z_2^{s(n)}$ of $T$, hence there are at most $2^{s(n)} - 1$ of them.

In addition, we show by an inductive argument that \linebreak[4]$\sum b_{2i}(N_j^T) \leq f_0(\dim N_j)$ and $\dim N_j < \frac{3}{4} n$ for all $N_j$. Putting this together, we deduce that $\sum b_{2i}(M^T) \leq f_0(n)$ by using the recursive formula for $f_0(n)$.

Finally note that, below dimension $54$, our symmetry assumption is sufficiently large that the classification results of Section \ref{Preliminaries} imply immediately that the upper bound of the Euler characteristic of $M^n$ is provided by $\chi(\cc\pp^{n/2}) = \frac n 2 + 1 = f_0(n)$.

\begin{proof}[Proof of Theorem \ref{LogAssumption}]
By Theorem \ref{LogAssumptionPLUS}, it suffices to show that $f_0(n) < 2^{3(\log_2 n)^2}$. This clearly holds for $n = 2$, so it suffices to prove it for $n \geq 4$. In this range, we actually prove the following, stronger estimate:
	\[f_0(n) \leq \of{\frac n 2 + 1}^{1 + \log_{4/3}\of{\frac n 2 + 1}}.\]

For $n \leq 52$, this is obvious since $f_0(n) = \frac n 2 + 1$. Proceeding inductively, for $n \geq 54$, we have
	\begin{eqnarray*}
	f_0(n)	
&=&\of{2^{s(n)} - 1} f_0\of{2\floor{\frac{3n-4}{8}}}\\
&\leq& \of{\frac{n(n+2)}{4} - 1} \of{\floor{\frac{3n-4}{8}} + 1}^{1 + \log_{4/3}\of{\floor{\frac{3n-4}{8}}+1}}\\
&<& \of{\frac{n}{2} + 1}^2 \of{\pfrac{3}{4} \of{\frac{n}{2}+1}}^{\log_{4/3}\of{\frac{n}{2} + 1}}\\
&=& \of{\frac{n}{2} + 1}^{1 + \log_{4/3}\of{\frac{n}{2} + 1}}.
	\end{eqnarray*}
\end{proof}

We proceed to the proof of Theorem \ref{LogAssumptionPLUS}. Recall that $M^n$ is a closed, positively curved manifold that admits an effective, isometric action by a torus $T$ with $\dim(T) \geq \log_{4/3} n$.

First, if $n < 22$, the bound on the dimension of $T$ implies that $\dim(T) > \frac{n}{2}$. As this contradicts the maximal possible symmetry rank (Theorem \ref{thm:GroveSearle}), the theorem is vacuously true in these dimensions. Moreover, the theorem continues to hold for dimensions $n\leq 52$ by Wilking's homotopy classification theorem (Theorem \ref{thm:Wilking}) and the definition of $f_0$. We proceed by induction to prove the result for $n\geq 54$.

Set $s = s(n)$, and choose any subgroup $\Z_2^s \subseteq T$ of involutions. We split the proof of the induction step into two cases.

\begin{description}
\item[Case 1] There exists $x \in M^T$ such that every $\iota\in\Z_2^s$ with $\dim\ker\of{T|_{M^\iota_x}} \leq 1$ and $\cod\of{M^\iota_x} \leq \frac{n}{2}$ actually has $\cod\of{M^\iota_x} \leq \frac{n}{4}$.
\item[Case 2] For all $x \in M^T$, there exists $\iota\in \Z_2^s$ with $\dim\ker\of{T|_{M^\iota_x}} \leq 1$ and $\frac{n}{4}<\cod\of{M^\iota_x} \leq \frac{n}{2}$.
\end{description}

If Case 1 occurs, we prove that $M$ has $4$--periodic rational cohomology. By Conner's theorem, we obtain the desired estimate
	\[\sum b_{2i}(M^T) \leq \sum b_{2i}(M) \leq \frac n 2 + 1 \leq f_0(n).\]

\begin{proof}[Proof in Case 1]
Fix some $x\in M^T$ as in the statement of Case 1. Let $0\leq j \leq \floor{\log_2 n} - 1$ be maximal such that there exists a $\Z_2^j \subseteq \Z_2^s$ with the property that every $\sigma\in\Z_2^j$ has $\dim\ker\of{T|_{M^\sigma_x}} \leq 1$ and $\cod\of{M^\sigma_x} \leq n/4$. We claim that $j = \floor{\log_2 n} - 1$. By Proposition \ref{pro:Case1lemma} and the comments above, this would suffice to show that $M$ is rationally $4$--periodic.

Suppose instead that $j \leq \floor{\log_2 n} - 2$. Choose a $\Z_2^{s-j}\subseteq \Z_2^{s}$ so that $\Z_2^j \cap \Z_2^{s-j} = \{\id\}$. Observe that
	\[s - j \geq s(n) - (\floor{\log_2 n} - 2) = \floor{\log_2(n+2)},\]
so Proposition \ref{pro:Griesmer} applies with $t=1$ and $c=0$ (alternatively, one can apply \cite[Lemma 1.8]{Kennard2}). Choose a nontrivial $\tau\in\Z_2^{s-j}$ with $\cod\of{M^\tau_x} \leq \frac{n}{2}$.

If $\dim\ker\of{T|_{M^\tau_x}} \geq 2$, then Proposition \ref{pro:dk2} implies that $H^*(M;\Q)$ is 4--periodic. Since this is our claim, we may assume that $\dim\ker\of{T|_{M^\tau_x}}\leq 1$ and hence, by the assumption in Case 1, that $\cod\of{M^\tau_x} \leq \frac{n}{4}$.

Fix for a moment any $\sigma\in\Z_2^j$. Observe that $M^{\sigma\tau}_x \supseteq M^\sigma_x\cap M^\tau_x$, so
	\[\cod\of{M^{\sigma\tau}_x} \leq \cod\of{M^\sigma_x \cap M^\tau_x} \leq \cod\of{M^\sigma_x} + \cod\of{M^\tau_x} \leq n/2.\]
By Proposition \ref{pro:dk2} again, we can conclude without loss of generality that $\dim\ker\of{T|_{M^{\sigma\tau}_x}} \leq 1$. But now the assumption in Case 1 implies that $\cod\of{M^{\sigma\tau}_x} \leq n/4$. Hence every $\iota$ in the $\Z_2^{j+1}$ generated by $\Z_2^j$ and $\tau$ satisfies $\dim\ker\of{T|_{M^\iota_x}} \leq 1$ and $\cod\of{M^\iota_x} \leq n/4$. This contradicts the maximality of $j$ and hence concludes the proof that $j \geq \floor{\log_2 n} - 1$.
\end{proof}

To conclude the proof of Theorem \ref{LogAssumptionPLUS}, it suffices to prove that $\sum b_{2i}(M^T) \leq f_0(n)$ in Case 2. In this case, we cannot calculate $H^*(M;\Q)$. Instead, we cover $M^T$ by at most $2^{s(n)}-1$ submanifolds, each of which satisfies the induction hypothesis. Adding together the estimates for each of these submanifolds and applying the recursive definition of $f_0$, we conclude the desired bound.

\begin{proof}[Proof in Case 2]
We use the notation $\beven(X) = \sum b_{2i}(X)$ for the sum of the even Betti numbers of a space $X$.

Suppose $M^T$ has $t$ components, and choose one point $x_i$ in each component. The assumption in Case 2 implies that there exist nontrivial involutions $\iota_1,\ldots,\iota_t\in \Z_2^s$ such that $\dim\ker\of{T|_{M^{\iota_j}_{x_j}}} \leq 1$ and $\frac{n}{2} \leq \dim\of{M^{\iota_j}_{x_j}} < \frac{3n}{4}$. Set $n_j = \dim(M^{\iota_j}_{x_j})$ for all $j$.

The $T$--action on $M$ restricts to a $T$ action on each $M^{\iota_j}_{x_j}$. After dividing by the kernel of the induced action, we obtain an effective action by a torus $T_j$ on $M^{\iota_j}_{x_j}$ with $\dim(T_j) \geq \dim(T) - 1$. Since $\dim(T) \geq \log_{4/3} n$ and $n > \pfrac{4}{3}n_j$, we see that $\dim\of{T_j} \geq \log_{4/3}(n_j)$. It follows from the induction hypothesis that $\beven\of{\of{M^{\iota_j}_{x_j}}^{T_j}} \leq f_0(n_j)$. But the components of $\of{M^{\iota_j}_{x_j}}^{T_j}$ are precisely the components of $M^T$ that also lie in $M^{\iota_j}_{x_j}$. Moreover, $n_j < \frac{3n}{4}$, so $n_j \leq 2\floor{\frac{3n-4}{8}}$, since both dimensions are even. Hence, we obtain
	\[\beven\of{\of{M^{\iota_j}_{x_j}}^T} = \beven\of{\of{M^{\iota_j}_{x_j}}^{T_j}} \leq f_0(n_j) \leq f_0\of{2\floor{\frac{3n-4}{8}}}.\]
We will use this estimate in a moment.

The key observation is that the set $\{\iota_1,\ldots,\iota_t\}\subseteq \Z_2^s\setminus\{\id\}$ has at most $2^s-1=2^{s(n)}-1$ elements. Moreover, since each $n_j \geq n/2$, we can conclude from Frankel's theorem that the set $\{M^{\iota_1}_{x_1},\ldots,M^{\iota_t}_{x_t}\}$ also has at most $2^{s(n)} - 1$ elements. Let $J\subseteq\{1,\ldots,t\}$ denote an index set in one-to-one correspondence with the set of $M^{\iota_j}_{x_j}$. Since every component $F \subseteq M^T$ is contained in some $\of{M^{\iota_j}_{x_j}}^T$ with $j\in J$, we have the estimate
	\[\beven(M^T)
		 = \sum_{F \subseteq M^T} \beven(F)
 		\leq \sum_{j\in J}\beven\of{\of{M^{\iota_j}_{x_j}}^T}.\]
Applying the estimate above and the recursive definition of $f_0$, we see that
	\[\beven(M^T) \leq \of{2^{s(n)} - 1} f_0\of{2\floor{\frac{3n-4}{8}}} = f_0(n).\]
\end{proof}

This finishes the proof of Theorem \ref{LogAssumptionPLUS} and hence of Theorem \ref{LogAssumption}.
\begin{rem}
We remark that our Case 1, Case 2 approach directly yields vanishing theorems for the elliptic genus as proven in \cite{Weisskopf}.
\end{rem}

We conclude this section with a remark on the size of $f_0(n)$.
\begin{rem}

Make the following definitions.
	\begin{itemize}
	\item Set $n_0=0, n_1 = 54$ and, for $i\geq 2$, let $n_i$ be the minimal integer satisfying $n_{i-1} \leq 2\floor{\frac{3n_i-4}{8}}$.
	\item Set $\kappa_0 = 1$ and, for $i\geq 0$, set
$\kappa_i = f_0(n_i)\left/\of{\frac{n_i}{2}+1}^{1 + \log_{4/3}\of{\frac{n_i}{2} + 1}}\right.$.
	\end{itemize}
See Table \ref{table01} for some approximate values of $n_i$ and $\kappa_i$.

\begin{table}[h]
\centering \caption{Some approximated values} \label{table01}
\begin{tabular}{l @{\hspace{8mm}}| c@{\hspace{8mm}} |@{\hspace{8mm}}l@{\hspace{8mm}}}
 $i=$ &$n_i=$ & $\kappa_i=$ \\
\hline
&&\\$0$ & $0$ & $1$ \\
$1$ & $54$ & $3.14823\cdot 10^{-15}$ \\
 $2$ & $74$ & $1.45259\cdot 10^{-15}$ \\
 $3$ & $100$ & $4.80780\cdot 10^{-16}$ \\
 $4$ & $135$ & $3.40869\cdot 10^{-16}$ \\
 $5$ & $183$& $1.10871\cdot 10^{-16}$ \\
 $6$ & $247$ & $2.15684\cdot 10^{-17}$ \\
\end{tabular}
\end{table}

By an argument similar to the one above that shows $f_0(n) < 2^{3(\log_2 n)^2}$, one could similarly show the following, better estimates for $f_0(n)$: For all $i\geq 1$ and $n\geq n_i$,
	\[f_0(n) \leq \kappa_i \of{\frac{n}{2} + 1}^{1 + \log_{4/3}\of{\frac{n}{2} + 1}}.\]

We will not pursue this here. However, we compare the function $f_0$ with the approximation given in Theorem \ref{LogAssumption} against the background of an exponential bound. For this note that it is easily possible to modify the arguments we provide in order to obtain the exponential function we draw in the pictures below as an alternative upper bound. From the pictures already it should be obvious why we did not spend more effort on this.

Already from the graphs in Figures \ref{fig1} and  \ref{fig2}, one sees that the functions grow strictly subexponentially. As for the asymptotic behavior, the exponential reference function may easily be chosen such that the power of two is a genuine fraction of $n$.
\begin{figure}[ht]
	\centering
  \includegraphics[width=0.47\textwidth]{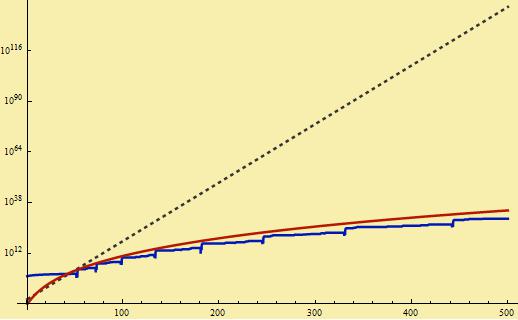}
  \includegraphics[width=0.47\textwidth]{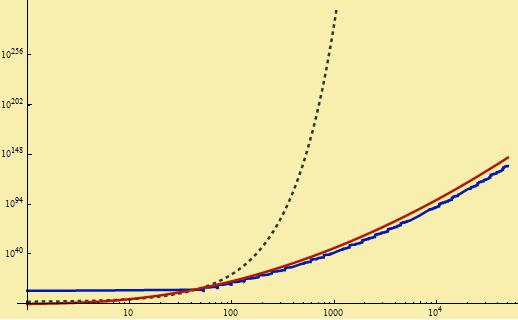}
	\caption{Comparison of the functions $f_0(n)$ (blue) and $\kappa_1\cdot(\frac{n}{2} + 1)^{1 + \log_{4/3}\of{\frac{n}{2} + 1}}$ (red)
under a logarithmically scaled $y$-axis respectively both axes being logarithmically scaled. As a reference the function $1.13576\cdot 10^{-12}\cdot 2^n$, which also bounds $f_0$, (black) is added.}
	\label{fig1}
\end{figure}
\begin{figure}[ht]
	\centering
  \includegraphics[width=0.47\textwidth]{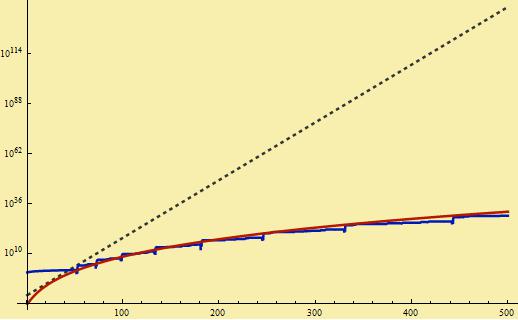}
  \includegraphics[width=0.47\textwidth]{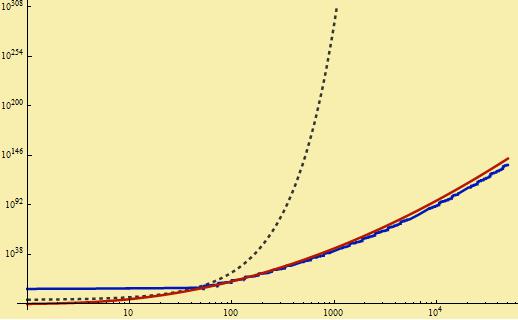}
	\caption{Comparison of the functions $f_0(n)$ (blue) and $\kappa_6\cdot(\frac{n}{2} + 1)^{1 + \log_{4/3}\of{\frac{n}{2}+1}}$ (red)
under a logarithmically scaled $y$-axis respectively both axes being logarithmically scaled. As a reference the function $1.13576\cdot 10^{-12}\cdot 2^n$, which also bounds $f_0$, (black) is added.}
	\label{fig2}
\end{figure}
\end{rem}

\bigskip
\section{Proof of Theorem B}\label{ProofOfLinearAssumptionDelta}
\bigskip

For integers $\alpha \geq 3$ and $n>0$, define
	\[s_\alpha(n) = \frac{n}{2\alpha} + 2\log_2\of{\frac{n}{2\alpha}} + \alpha + 3.\]
Theorem \ref{LinearAssumptionDelta} will follow from the following theorem.
\begin{theo}\label{LinearAssumptionAlpha}
Let $\alpha\geq 3$ be an integer, and let $M^n$ be a closed, simply connected Riemannian manifold with positive sectional curvature. Assume $T$ is a torus acting isometrically and effectively on $M$ with $\dim(T)  \geq s_\alpha(n)$. The following hold:
	\begin{enumerate}
	\item $b_0(M^T) \leq a_\alpha n + 1$, where $a_\alpha = 3\cdot 2^{\alpha - 4}/\alpha$.
	\item $\displaystyle \chi(M)\leq\sum b_{2i}(M^T) \leq \of{\frac{n}{2} + 1} \of{1 + \log_{b_\alpha}\of{\frac{n}{2} + 1}}$ when $\alpha \geq 4$, where $b_\alpha = \frac{2^{\alpha-3}}{2^{\alpha-3}-1}$. When $\alpha = 3$, the upper bound is $\frac n 2 + 1$.
	\end{enumerate}
\end{theo}

We show now that this theorem implies Theorem \ref{LinearAssumptionDelta}.
\begin{proof}[Proof of Theorem \ref{LinearAssumptionDelta}]
Let $\delta > 0$, and assume $M^n$ is a closed, positively curved Riemannian manifold with symmetry rank $r > \delta n$.

Choose an integer $\alpha \geq 3$ with $\delta > \frac{1}{2\alpha}$, and choose another integer $n_0$ such that $n\geq n_0$ implies $\delta n > s_\alpha(n)$. Finally, choose a constant $c$ such that all of the following hold:
	\begin{itemize}
	\item For all $m\leq n_0$, every non-negatively curved $m$--manifold $N$ has $\sum b_i(N) \leq c$.
	\item $a_\alpha n + 1 \leq c n$ for all $n \geq n_0$.
	\item $2\log_{b_\alpha}(2) \leq c$.
	\end{itemize}
Observe that the first is possible by Gromov's Betti number estimate.

It now follows from either Gromov's theorem (if $n \leq n_0$) or Theorem \ref{LinearAssumptionAlpha} (if $n > n_0$) that $b_0(M^T) \leq c n$ and $\chi(M) < c n \log_2 n$.
\end{proof}

We spend the rest of this section proving Theorem \ref{LinearAssumptionAlpha}. Our task is threefold. First we prove Lemma \ref{LinearAssumptionBaseCase}, which implies the $\alpha = 3$ case of Theorem \ref{LinearAssumptionAlpha}. Using this as the base case for an induction over $\alpha$, we prove the two parts of Theorem \ref{LinearAssumptionAlpha} in Theorems \ref{LinearAssumptionPart1} and \ref{LinearAssumptionPart2}, respectively.

\begin{lemma}\label{LinearAssumptionBaseCase}
Assume $M^n$ is a closed, one-connected Riemannian manifold with positive sectional curvature. If a torus $T$ acts effectively by isometries on $M$ with
	\[\dim(T) \geq \frac{n}{6} + 2\log_2(n) - 3,\]
then
	\begin{enumerate}
	\item $M$ is homotopy equivalent to $\s^n$ or $\cc\pp^{n/2}$,
	\item $M$ has the integral cohomology of $\hh\pp^{n/4}$,
	\item $n\equiv 2\bmod{4}$ and, for all fields $\mathbb{F}$, the cohomology of $M$ with coefficients in $\mathbb{F}$ is isomorphic to either $H^*(\cc\pp^{n/2};\mathbb{F})$ or $H^*(\s^2\times\hh\pp^{\frac{n-2}{4}};\mathbb{F})$, or
	\item $n\equiv 3\bmod{4}$ and, for all fields $\mathbb{F}$, the cohomology of $M$ with coefficients in $\mathbb{F}$ is isomorphic to either $H^*(\s^n;\mathbb{F})$ or $H^*(\s^3\times\hh\pp^{\frac{n-3}{4}})$.
	\end{enumerate}
In particular, if $\dim(T) \geq s_3(n)$, then $M$ has $4$--periodic cohomology and $\sum b_{2i}(M^T) \leq \sum b_{2i}(M) \leq \frac{n}{2} + 1$.
\end{lemma}

Note that the conclusion is the same as in Wilking's cohomology classification (Theorem \ref{thm:Wilking}) for $n \geq 6000$. Using Wilking's proof, we obtain the same conclusion in all dimensions using the slightly larger symmetry rank bound.

\begin{proof}[Proof of Lemma \ref{LinearAssumptionBaseCase}]
We induct over dimension. For $n\leq 3$, the theorem is trivial. For $4\leq n \leq 16$, $\dim(T) \geq n/2$, hence the result follows from the Grove and Searle's diffeomorphism classification (Theorem \ref{thm:GroveSearle}). For $17\leq n\leq 120$, the symmetry rank is at least $\dim(T) \geq \frac{n}{4} + 1$, so the theorem follows from Wilking's homotopy classification (Theorem \ref{thm:Wilking}).

Suppose now that $n > 120$. The assumption on $T$ implies
	\[\dim(T) \geq \max\left(\frac{n}{6} + 1, \frac{n}{8} + 14\right).\]
By Proposition 8.1 in \cite{Wilking03}, if there exists an involution $\iota\in T$ with $\cod(M^\iota) \leq \frac{7n}{24}$, then the conclusion of the lemma holds. It therefore suffices to prove the existence of such an involution.

To see that such an involution exists, we adapt the proof of Corollary 3.2.a in \cite{Wilking03}. Let $T$ denote the torus acting effectively and isometrically on $M$. By Berger's theorem, there exists $T^r\subseteq T$ such that $T^r$ has a fixed point $x\in M$, where $r = \dim(T) - \delta$ and where $\delta = 0$ if $n$ is even and $\delta = 1$ if $n$ is odd. We consider the isotropy representation $T^{r} \to \SO(T_x M) = \SO(n)$ at $x$. The dimension of the $(-1)$--eigenspace of the image of an involution $\iota\in T^{r}$ in $\SO(n)$ is equal to the codimension of $M^\iota_x$. The proof of Corollary 3.2.a in \cite{Wilking03} states that such an nontrivial involution $\iota\in T^r$ exists with $\cod(M^\iota_x) \leq 7n/24$ if
	\[r \geq 0.325973 m + \frac{3}{2} \log_2(m) + 2.6074\]
where $m = \floor{n/2}$. The original proof concludes that this is the case for $r\geq\frac{m}{3} + 1$ and $m\geq 2700$, and in particular for $r\geq \frac{n}{6} + 1$ and $n\geq 6000$. In our case, the assumption on $r$ in terms of $m$ is that
	\[r = \dim(T) -\delta \geq \frac{m}{3} + 2\log_2(m) - 5 - \delta.\]
Except for $n = 126$ and $n=131$, this estimate holds for all $n > 120$. For $n = 126$ and $n = 131$, one can instead apply the Griesmer bound to find the desired involution. In all cases therefore, if $n > 120$, the involution $\iota$ with $\cod(M^\iota_x) \leq 7n/24$ exists. This concludes the proof of the lemma.
\end{proof}

With the base case ($\alpha = 3$) complete, we pause for a moment to prove a couple of lemmas that will facilitate the induction over $\alpha$. The third estimate of Lemma \ref{salphaEstimates} together with Lemma \ref{LinearAssumptionLemma2} already suggest how this will work.

\begin{lemma}\label{salphaEstimates}
For $\alpha \geq 4$, $s_\alpha(n)$ satisfies the following:
	\begin{enumerate}
	\item If $n \leq \alpha(\alpha-1)/3$, then $s_\alpha(n) \geq s_{\alpha-1}(n)$.
	\item If $n > \alpha(\alpha-1)/3$, then $s_\alpha(n) \geq \log_2(a_\alpha n^2 + 2n + 2)$.
	\item If $k\geq n/\alpha$, then $s_\alpha(n) - 1 \geq s_{\alpha-1}(n-k)$.
	\end{enumerate}
\end{lemma}

\begin{proof}
The estimate $s_\alpha(n) \geq s_{\alpha - 1}(n)$ is equivalent to
	\[1 \geq \frac{n}{2\alpha(\alpha - 1)} + 2\log_2\pfrac{\alpha}{\alpha - 1}.\]
The logarithm term is decreasing in $\alpha$, so it is at most $2\log_2\pfrac{4}{3} < \frac{5}{6}$. In particular, $s_\alpha(n) \geq s_{\alpha-1}(n)$ holds if $1 \geq \frac{n}{2\alpha(\alpha-1)} + \frac{5}{6}$, which is to say if $n \leq \alpha(\alpha-1)/3$.

For the second estimate, suppose that $n > \alpha(\alpha-1)/3$. Observe that $\alpha \geq 4$ implies
	\[2^{\frac{n}{2\alpha} + 2} > 2^{\frac{\alpha - 1}{6}+2} > \alpha,\]
hence
	\[2^{s_\alpha(n)}
	= 2^{\frac{n}{2\alpha}+\alpha+3}\pfrac{n}{2\alpha}^2
	> \alpha \cdot 2^{\alpha+1}\pfrac{n}{2\alpha}^2
	= \frac{2^{\alpha-1} n^2}{\alpha}.\]
On the other hand, $\alpha\geq 4$ and $n > \alpha(\alpha-1)/3$ imply that
	\[2n + 2
	<   \pfrac{\alpha(\alpha-1)}{3} n
	  + \frac{1}{4}\pfrac{\alpha(\alpha-1)}{3}^2
	< \frac{5}{4} n^2 \leq \frac{5}{8}\pfrac{2^{\alpha-1}n^2}{\alpha}.\]
Combining this estimate with the one above for $s_\alpha(n)$, we conclude the proof of the second estimate as follows:
	\[a_\alpha n^2 + 2n + 2
	= \frac{3}{8}\pfrac{2^{\alpha-1}n^2}{\alpha} + (2n+2)
	< \frac{2^{\alpha-1}n^2}{\alpha} < 2^{s_\alpha(n)}.\]

Finally, if $k \geq n/\alpha$, then $n - k\leq \frac{\alpha-1}{\alpha} n$. Hence we conclude the third estimate as follows:
	\begin{eqnarray*}
	s_{\alpha - 1}(n-k)
	& = & \frac{n-k}{2(\alpha-1)} + 2\log_2\pfrac{n-k}{2(\alpha-1)}
			+ (\alpha - 1) + 3\\
	&\leq& \frac{n}{2\alpha} + 2\log_2\pfrac{n}{2\alpha} + \alpha+3-1\\
	&=& s_\alpha(n) - 1.
	\end{eqnarray*}
\end{proof}

\begin{lemma}\label{LinearAssumptionLemma2}
Assume $M^n$ is a closed, one-connected Riemannian manifold with positive sectional curvature, and assume $T$ is a torus acting effectively and isometrically on $M$ with $\dim(T) \geq s_\alpha(n)$ and $\alpha\geq 4$. If there exists a nontrivial involution $\iota\in T$ and a point $x\in M^\iota$ such that $\cod(M^\iota_x) \leq \frac{n}{\alpha}$ or such that $\cod(M^\iota_x) \leq \frac{n}{2}$ and $\dim\ker\of{T|_{M^\iota_x}} \geq 2$, then $M$ has $4$--periodic rational cohomology.
\end{lemma}
Here, we are using the notation from Section \ref{ProofOfLogAssumption} that $\dim\ker\of{T|_{M^\iota_x}}$ denotes the dimension of the kernel of the induced $T$--action on $M^\iota_x$.
\begin{proof}
Choose $\iota$ and $x$ as in the assumption of the lemma. If $k=\cod(M^\iota_x)$ is at most $n/\alpha$, then Proposition \ref{pro:new4per} implies that $M$ has $4$--periodic rational cohomology. Indeed, the proposition applies since $\alpha \geq 4$ implies that $\alpha + 3 - 2\log_2(2\alpha) \geq 1$ and hence that
	\[\dim(T) \geq s_\alpha(n) \geq \frac{k}{2} + 2\log_2(n) + 1.\]

On the other hand, if $k \leq n/2$ and $\dim\ker\of{T|_{M^\iota_x}} \geq 2$, then we are in the situation of Proposition \ref{pro:dk2}, so we conclude again that $H^*(M;\Q)$ is 4--periodic.
\end{proof}

We are ready to prove the first part of Theorem \ref{LinearAssumptionAlpha}:

\begin{theo}[Theorem \ref{LinearAssumptionAlpha}, Part 1]\label{LinearAssumptionPart1}
Let $M^n$ be a closed, one-connected Riemannian manifold with positive sectional curvature, and assume $T$ is a torus acting effectively by isometries on $M$. If $\alpha\geq 3$ and $\dim(T)\geq s_\alpha(n)$, then $b_0(M^T) \leq a_\alpha n + 1$ where $a_\alpha = 3\cdot 2^{\alpha-4}/\alpha$.
\end{theo}

\begin{proof}
By Lemma \ref{LinearAssumptionBaseCase}, the theorem follows in the base case of $\alpha = 3$. We proceed by induction the prove the theorem for $\alpha\geq 4$.

By Lemma \ref{salphaEstimates} and the induction hypothesis, we may we may assume that $n > \alpha(\alpha-1)/3$. In particular, by Lemma \ref{salphaEstimates} again, we may assume that
	\[s_\alpha(n) \geq \log_2(a_\alpha n^2 + 2n + 2).\]

Next observe that if $M$ has $4$--periodic rational cohomology, then the even Betti numbers of $M$ are at most one. Hence we can conclude the theorem using Conner's theorem:
	\[b_0(M^T) \leq \sum b_{2i}(M^T) \leq \sum b_{2i}(M) \leq \frac{n}{2} + 1.\]
In particular, by Lemma \ref{LinearAssumptionLemma2}, we may assume that, for all nontrivial involutions $\iota\in T$ and for all $x\in M^\iota$, if $\cod M^\iota_x \leq n/2$, then $\cod M^\iota_x > n/\alpha$ and $\dim\ker\of{T|_{M^\iota_x}} \leq 1$.

We proceed by contradiction. Assume there exist $t' = \floor{a_\alpha n} + 2$ distinct components of $M^T$.
The estimate on $s_\alpha(n)$ established at the beginning of the proof implies that \[\dim(T) \geq s_\alpha(n) \geq \log_2(t'n + 2).\]
By Proposition \ref{pro:Griesmer} (with $c = 0$), there exists a nontrivial involution $\iota\in T$ and a component $N\subseteq M^\iota$ with $\cod(N)\leq\frac{n}{2}$ such that $N$ contains at least $\ceil{\frac{t'+1}{2}}$ of the $t'$ components of $M^T$. As established above, we may assume that $\cod(N) > n/\alpha$ and $\dim\ker\of{T|_{M^\iota_x}} \leq 1$.

Let $\overline{T} = T/\ker(T|_N)$, and observe that $\overline{T}$ acts effectively and isometrically on $N$. Moreover, the components of $N^{\overline{T}}$ are the components of $M^T$ that are contained in $N$. In particular, the number of components of $N^{\overline{T}}$ is at least $\ceil{\frac{t'+1}{2}}$. On the other hand, since
	\[\dim(\overline{T})\geq \dim(T) - 1 \geq s_\alpha(n) - 1 \geq s_{\alpha-1}\of{\dim N}\]
by Lemma \ref{salphaEstimates}, the induction hypothesis implies that the number of components of $N^{\overline{T}}$ is at most $a_{\alpha-1} \dim(N) + 1$. Putting these estimates together, we conclude that
	\[\frac{t'+1}{2} \leq a_{\alpha-1} \dim(N) + 1 \leq \pfrac{3 \cdot 2^{\alpha-5}}{\alpha-1}\pfrac{\alpha-1}{\alpha} n + 1,\]
and hence that $t' \leq a_\alpha n + 1$. This contradicts our definition of $t'$, so the proof is complete.
\end{proof}

With the proof of Theorem \ref{LinearAssumptionPart1} complete, it suffices to prove the following, which is the the second part of Theorem \ref{LinearAssumptionAlpha}.
\begin{theo}[Theorem \ref{LinearAssumptionAlpha}, Part 2]\label{LinearAssumptionPart2}
Let $M^n$ be a closed, one-connected Riemannian manifold with positive sectional curvature, and assume $T$ is a torus acting effectively by isometries on $M$. If $\alpha\geq 3$ and $\dim(T) \geq s_\alpha(n)$, then
	\[\chi(M)\leq\sum b_{2i}(M^T) \leq \of{\frac{n}{2} + 1}\of{1 + \log_{b_\alpha}\of{\frac{n}{2}+1}}\]
where $b_\alpha = 2^{\alpha-3}/(2^{\alpha-3} - 1)$.
\end{theo}

The first step of the proof is to find submanifolds that have $4$--periodic cohomology and that cover multiple components of $M^T$. The second step is to use these submanifolds together with Conner's theorem to conclude the upper bound on $\beven(M^T)$. To quantify the first step, we prove the following:

\begin{lemma}[Chains with periodic tails]\label{lem:LinearAssumptionPart2chains}
Let $M^n$ be a closed, one-connected Riemannian manifold with positive sectional curvature, and assume $T$ is a torus acting effectively and isometrically on $M$. If $\alpha \geq 3$ and $\dim(T)\geq s_\alpha(n)$, then there exist $0\leq j \leq \alpha - 3$ and a chain $M = M_0^{n_0} \supseteq \cdots \supseteq M_j^{n_j}$ of submanifolds satisfying
	\begin{enumerate}
	\item every $M_i$ contains $t_i$ components of $M^T$ for some $t_i \geq \frac{t-1}{2^i} + 1$,
	\item $T$ acts on every $M_i$ and $\dim\of{T/\ker(T|_{M_i})}\geq s_{\alpha-i}(n_i)$, and
	\item $M_j$ has $4$--periodic rational cohomology.
	\end{enumerate}
\end{lemma}

\begin{proof}[Proof of Lemma \ref{lem:LinearAssumptionPart2chains}]
Clearly the chain $M = M_0$ of length 0 satisfies properties (1) and (2). Let $M=M_0^{n_0} \supseteq \cdots \supseteq M_j^{n_j}$ denote a maximal chain satisfying properties (1) and (2).

First note that the maximality of $j$ implies $n_j > (\alpha-j)(\alpha-j-1)/3$. Indeed if this were not the case, then we could define $M_{j+1} = M_j$ and conclude from Lemma \ref{salphaEstimates} that the chain $M_0 \supseteq\cdots\supseteq M_{j+1}$ satisfies properties (1) and (2).

Next note that, if $j\geq \alpha-3$, then the subchain $M=M_0\supseteq\cdots\supseteq M_{\alpha-3}$ has properties (1)--(3) since
	\[\dim(T/\ker(T|_{M_{\alpha-3}})) \geq s_{3}(\dim M_{\alpha-3}),\]
which implies that $M_{\alpha-3}$ has $4$--periodic rational cohomology by Lemma \ref{LinearAssumptionBaseCase}.

We may assume therefore that $j\leq \alpha-4$. We show in this case that $M_j$ has $4$--periodic rational cohomology. We will do this by choosing an involution $\iota\in T$ and a component $M_{j+1}$ of $M_j^\iota$, using the maximality of $j$ to conclude that $M_{j+1}$ does not satisfy properties (1) and (2), then using this information to conclude that $M_j$ has $4$--periodic rational cohomology.

Our first task is to choose the involution $\iota$. Set $\overline{T} = T/\ker(T|_{M_j})$. Note that $\overline{T}$ is a torus acting effectively on $M_j$ with dimension at least $s_{\alpha - j}(n_j)$ by property (2). By property (1) together with the first part of Theorem \ref{LinearAssumptionPart1}, we have
	\[t_j \leq b_0\of{(M_j)^{\overline{T}}} \leq a_{\alpha-j}n_j + 1.\]
In particular,
	\[\log_2(t_jn_j + 2) \leq \log_2(a_{\alpha-j}n_j^2 + n_j + 2) \leq s_{\alpha-j}(n_j),\]
where the last inequality holds by Lemma \ref{salphaEstimates} and our assumption that $n_j>(\alpha-j)(\alpha-j-1)/3$. This estimate allows us to apply Proposition \ref{pro:Griesmer} (with $t = t_j$ and $c=0$). Choose a nontrivial involution $\iota \in \overline{T}$ and a component $M_{j+1}$ of $M_j^\iota$ with $\cod(M_{j+1}\subseteq M_j)\leq \frac{n_j}{2}$ such that $M_{j+1}$ covers at least $\frac{t_j+1}{2}$ of the $t_j$ components of $M_j^{\overline{T}} = M_j^T$.

In particular, property (1) and our choice of $M_{j+1}$ imply that $M_{j+1}$ covers at least
	\[\frac{t_j+1}{2} \geq \frac{1}{2}\of{\of{\frac{t-1}{2^j} + 1} + 1} = \frac{t-1}{2^{j+1}} + 1\]
of the components of $M^T$, hence the chain $M = M_0 \supseteq\cdots\supseteq M_{j+1}$ satisfies property (1). By maximality of $j$, this chain must fail to satisfy property (2). It follows that
	\[\cod(M_{j+1} \subseteq M_j) < \frac{\dim M_j}{\alpha-j}\]
or
	\[\dim(T/\ker(T|_{M_{j+1}})) < \dim(T/\ker(T|_{M_j})) - 1.\]
If $\cod(M_{j+1}\subseteq M_j) < \frac{\dim M_j}{\alpha-j}$, then $M_j$ has $4$--periodic rational cohomology by \PROP{pro:new4per}, and if $\dim(T/\ker(T|_{M_{j+1}})) \leq \dim(T/\ker(T|_{M_j})) - 2$, then it follows as in the proof of Proposition \ref{pro:dk2} that $M_j$ is 4--periodic. This concludes the proof of the existence of a chain $M=M_0\supseteq \cdots \supseteq M_j$ satisfying properties (1), (2), and (3).
\end{proof}

With the first step complete, the second step of the proof of Theorem \ref{LinearAssumptionPart2} is to cover $M^T$ by chains $M=M_0 \supseteq\cdots\supseteq M_j$ of varying lengths $j \leq \alpha-3$ satisfying properties (1), (2), and (3). We will see that at most $1 + \log_{b_\alpha}\of{\frac{n}{2}+1}$ chains are required. Since the $M_j$ in each such chain has 4--periodic rational cohomology, it satisfies
	\[\beven(M_j) \leq \frac{\dim M_j}{2} + 1 \leq \frac{n}{2} + 1.\]
Combining these facts, Theorem \ref{LinearAssumptionPart2} -- and hence Theorem \ref{LinearAssumptionAlpha} -- follows. We proceed with the details.

\begin{proof}[Proof of Theorem \ref{LinearAssumptionPart2}]
First, by the claim, there exists a submanifold $N_1\subseteq M$ with $4$--periodic rational cohomology that covers at least $\frac{t-1}{2^{\alpha-3}} + 1$ of the $t$ components of $M^T$. This leaves $u_1$ components of $M^T$ uncovered by $N_1$, where
	\[u_1 \leq t - \left(\frac{t-1}{2^{\alpha-3}} + 1\right) = \of{1 - \frac{1}{2^{\alpha-3}}}(t-1) = (t-1)/b_{\alpha},\]
where $b_{\alpha} = \frac{2^{\alpha-3}}{2^{\alpha-3} - 1}$.

Next, we apply the claim to the remaining $u_1$ components of $M^T$. This yields a submanifold $N_2\subseteq M$ with $4$--periodic rational cohomology such that $N_1\cup N_2$ covers all but $u_2$ components of $M^T$, where
	\[u_2 \leq (u_1-1)/b_{\alpha} \leq (t-1-b_\alpha)/b_\alpha^2.\]
Continuing in this way, we claim that we obtain a cover of $M^T$ by rationally $4$--periodic submanifolds $N_1,\ldots,N_h\subseteq M$ where
	\[h \leq \log_{b_\alpha}\of{(b_\alpha-1)t+1}.\]
Indeed, after $h$ steps of this process, there are $u_h$ components of $M^T$ not yet covered, and
	\[u_h \leq (t - 1 - b_\alpha - b_\alpha^2 - \ldots - b_\alpha^{h-1})/b_\alpha^{h}.\]
If the right-hand side is less than 1, then we have covered $M^T$ by the submanifolds $N_1,\ldots,N_h$. Moreover, the right-hand side is less than one if and only if
		\[t < 1 + b_\alpha + \ldots + b_\alpha^{h} = \frac{b_\alpha^{h+1}-1}{b_\alpha - 1},\]
which in turn holds if and only if
	\[h \geq \floor{\log_{b_\alpha}((b_\alpha-1)t + 1)}.\]
Taking $h = \floor{\log_{b_\alpha}((b_\alpha-1)t + 1)}$, we see that we can cover $M^T$ by rationally $4$--periodic submanifolds $N_1,\ldots,N_h\subseteq M$.

Applying this fact together with Conner's theorem, we have
	\[\beven(M^T) = \sum_{F\subseteq M^T} \beven(F) \leq \sum_{i=1}^h \beven(N_i^T) \leq \sum_{i=1}^h \beven(N_i).\]
Recall that we have $\beven(N_i) \leq \frac{n}{2} + 1$ for all $i$, hence, by using the first part of the theorem, we can conclude the theorem by estimating $h$ as follows:
	\begin{eqnarray*}
	h	&\leq& 	\log_{b_\alpha}\of{(b_\alpha-1)(a_\alpha n + 1)+1}\\
		&=&		\log_{b_\alpha}\of{\pfrac{1}{2^{\alpha-3}-1}\pfrac{3\cdot 2^{\alpha - 4}}{\alpha} n + b_\alpha}\\
		&=&		1 + \log_{b_\alpha}\of{\frac{3n}{2\alpha} + 1}\\
		&<& 		1 + \log_{b_\alpha}\of{\frac{n}{2} + 1}.
	\end{eqnarray*}
\end{proof}

\bigskip
\section{On bounding the Euler characteristic from below}\label{sec:OddBettis}
\bigskip

In this section, we prove the following, which contains Theorem \ref{thm:OddBettis}:

\begin{theo}\label{thm:OddBettisPLUS}
Let $M^n$ be an even-dimensional, one-connected, closed Riemannian manifold with positive sectional curvature. If a torus $T$ acts effectively by isometries on $M$ with $\dim(T) \geq 2\log_2(n) + 2$, then $\sum b_{2i+1}(M^T) = 0$ and $\chi(M) \geq 2$ if any one of the following holds:
	\begin{enumerate}
	\item\label{thm:OddBettiAss2} $n\equiv 0 \bmod{4}$.
	\item\label{thm:OddBettiAss3} $b_2(M)$, $b_3(M)$, or $b_4(M)$ is zero.
	\item\label{thm:OddBettiAss4} the Bott--Grove--Halperin conjecture holds.
	\end{enumerate}
If, in addition, the torus acts equivariantly formally, then the odd Betti numbers of $M$ vanish and $\sum b_{2i}(M)=\chi(M)$.
\end{theo}

We remark that the assumption $\dim(T) \geq 2\log_2(n)-2$ suffices in the first case. Indeed, it follows from \cite[Theorem A]{Kennard1} that $\sum b_{2i+1}(M^T) = 0$, hence the argument below implies  that $\chi(M) \geq 2$. We proceed to the proof, which will require the rest of this section.

First, the result on equivariant formality follows directly from the properties above it. We note that, if we make the slightly stronger assumption that $\dim(T) \geq \log_{4/3}(n)$, then this theorem and Theorem \ref{LogAssumption} imply two-sided bounds on the Betti numbers of $M$.

Second, we use the following theorem (see \cite[Corollary IV.2.3, p.~178]{Bredon72}):
\begin{theo}\label{ExactlyOne}
If $X$ is a compact, closed, orientable manifold, then no torus action on $X$ can have exactly one fixed point.
\end{theo}

Assume for a moment that each of Assumptions (\ref{thm:OddBettiAss2})--(\ref{thm:OddBettiAss4}) imply that $\sum b_{2i+1}(M^T) = 0$. Since $\chi(M) = \chi(M^T)$ by Conner's theorem, and since $M^T$ is nonempty by Berger's theorem, it follows that $\chi(M)>0$. Moreover, since $M$ is orientable, Theorem \ref{ExactlyOne} implies $M^T$ is not a single point. Hence there are at least two components or one of positive, even dimension. Since the components of $M^T$ are orientable, we have in either case that $\chi(M)\geq 2$.

This leaves us with the task of proving $\sum b_{2i+1}(M^T) = 0$ under each of Assumptions (\ref{thm:OddBettiAss2})--(\ref{thm:OddBettiAss4}). Fix any component $F$ of $M^T$. Taking $c = 6$ in Theorem \ref{thmP}, we conclude the existence of a submanifold $N\subseteq M$ such that
	\begin{enumerate}[(a)]
	\item\label{prop1} $N$ is rationally $4$--periodic,
	\item\label{prop2} $N\subseteq M$ is $6$--connected,
	\item\label{prop3} $\dim(N) \equiv \dim(M) \bmod{4}$,
	\item\label{prop4} $\dim(N) \geq 10$,
	\item\label{prop5} $N$ is totally geodesic in $M$, and
	\item\label{prop6} $T$ acts on $N$ and $F$ is a component of $N^T$.
	\end{enumerate}
The last two points follow from the choice of $N$ as a component of $M^H$ for some $H\subseteq T$. Indeed, such an $N$ is totally geodesic, and $T$ acts on $N$ since $T$ is abelian and connected.

By (\ref{prop6}) and Conner's theorem,
	\[\sum b_{2i+1}(F) \leq \sum b_{2i+1}(N^T) \leq \sum b_{2i+1}(N),\]
hence it suffices to show that $N$ has vanishing odd Betti numbers. Moreover, by (\ref{prop1}) and (\ref{prop2}), $N$ is simply connected and has 4--periodic Betti numbers, so it suffices to prove that $b_3(N) = 0$. We prove this in three cases, which together cover Assumptions (\ref{thm:OddBettiAss2})--(\ref{thm:OddBettiAss4}).

First, assume $n\equiv 0\bmod{4}$. By (\ref{prop3}), $\dim(N)$ is also divisible by four, hence four--periodicity and Poincar\'e duality imply
		\[b_3(N) = b_7(N) = \cdots = b_{\dim(N) - 1}(N) = b_1(N) = 0.\]

Second, assume $n\equiv 2\bmod{4}$, and assume $b_2(M)$, $b_3(M)$, or $b_4(M)$ is zero. By (\ref{prop2}), this is equivalent to $b_2(N)$, $b_3(N)$, or $b_4(N)$ vanishing. We prove that
		\[b_2(N) = 0~\implies~b_4(N) = 0~\implies~b_3(N)=0.\]
	The first implication holds because four--periodicity, Poincar\'e duality, and the assumption that $\dim(N)\equiv n\equiv 2\bmod{4}$ imply $b_2(N) = b_4(N)$. The second implication follows from the definition of four--periodicity. Indeed, the element $x\in H^4(N;\Q)$ inducing periodicity is zero, hence (\ref{prop4}) and four--periodicity implies that the map $H^3(N;\Q) \to H^7(N;\Q)$ induced by multiplication by $x=0$ is an isomorphism. This can only be if $b_3(N) = 0$.

Finally, assume $n\equiv 2\bmod{4}$, and assume that the Bott--Grove--Halperin conjecture holds. Note that, by the previous case, we may assume that $b_2(N) = b_4(N) = 1$. By (\ref{prop5}), $N$ has positive curvature, so the Bott--Grove--Halperin conjecture implies that $N$ is rationally elliptic. In particular, $\chi(N) \geq 0$. Using four--periodicity and the condition that $\dim(N)\equiv n\equiv 2\bmod{4}$, we can calculate $\chi(N)$ in terms of the $b_3(N)$:
	\[0 \leq \chi(N) = 2 + \pfrac{\dim(N) - 2}{4}\of{2 - b_3(N)}.\]
By (\ref{prop4}), this inequality implies $b_3(N) \leq 3$. But $b_3(N)$ must be even by Poincar\'e duality and the graded commutativity of the cup product, hence $b_3(N) \leq 2$. Returning to the expression above for $\chi(N)$, we conclude that $\chi(N) > 0$, and returning to the assumption of rational ellipticity, we conclude that $b_3(N) = 0$, as required.

\bigskip
\section{Proofs of Corollaries C and D}\label{Corollaries}
\bigskip

In this section we prove Corollary \ref{cor:StableHopf} and comment on some consequences of it. We then state and prove a more elaborate version of Corollary \ref{cor:SymmetricSpaces}.

\begin{proof}[Proof of Corollary \ref{cor:StableHopf}]
Recall that we have a closed, one-connected manifold $M^n$ with positive sectional curvature and symmetry rank $\log_{4/3} n$. We assume that $M$ is a $k$--fold Cartesian product or a $k$--fold connected sum of a manifold $N$. We claim that $k < 3(\log_2 n)^2$ or $k < 2^{3(\log_2 n)^2}$, respectively. We prove the two cases simultaneously.

Observe that $\log_{4/3} n \geq 2\log_2 n$, hence $b_2(M) \leq 1$ by Theorem \ref{thmP}. Since $b_2(M) = kb_2(N)$ in both cases, and since we may assume without loss of generality that $k\geq 2$, we conclude that $b_2(M) = 0$. This implies $\chi(M) \geq 2$ by Theorem \ref{thm:OddBettis}.

Theorem \ref{LogAssumption} implies that $\chi(M)\leq 2^{3(\log_2 n)^2}$. Using the multiplicativity and additivity formulas for the Euler characteristic, we obtain either
	\[2 \leq \chi(N)^k < 2^{3(\log_2 n)^2}\]
or
	\[2 \leq 2 + k(\chi(N) - 2) < 2^{3(\log_2 n)^2}\]
in the respective cases.

These inequalities imply either
	\[1 < \chi(N) < 2^{\frac{3(\log_2 n)^2}{k}}\]
or
	\[2 \leq \chi(N) < 2 + \frac{2^{3(\log_2 n)^2}}{k}.\]
The first is a contradiction for $k \geq 3(\log_2 n)^2$, while the second implies $\chi(N) = 2$ for $k \geq 2^{3(\log_2 n)^2}$, another contradiction.
\end{proof}

\pagebreak[2]

\begin{rem}~
\begin{itemize}
\item Note that if there exists a non-negatively curved metric on $M$, also the product metric is non-negatively curved. No similar result on the existence of non-negatively curved metrics on connected sums seems to be known unless it is the sum of two rank one symmetric spaces (see \cite{Cheeger73}). Thus the result on the Cartesian power, i.e.~the first assertion in the stable Hopf conjecture with symmetry (see Corollary \ref{cor:StableHopf}), appears to be stronger than the second one.

\item As for the assertion on the iterated fibrations, the proof is similar. Since the $F_i=G_i/H_i$ are compact homogeneous spaces of equal rank Lie groups, they have positive Euler characteristic and their Betti numbers are concentrated in even degrees. Hence $\chi(F_i)\geq 2$ for all $i$ by Poincar\'e duality, and the arguments in the proof of Corollary \ref{cor:StableHopf} apply equally.
\end{itemize}
\end{rem}
\begin{ex}
Let us make the result provided by Corollary \ref{cor:StableHopf} more precise in the classical case of a product with factors $\s^2$.
For example, the estimates in Theorem \ref{LogAssumption}
imply that there is no metric of positive curvature on any of $(\s^2)^{\times 124}$, $(\s^2)^{\times 125}$, \dots, $(\s^2)^{\times 314}$ compatible with the isometric action of a $20$--torus.
\end{ex}

We now proceed to Corollary \ref{cor:SymmetricSpaces}.

\begin{proof}[Proof of Corollary \ref{cor:SymmetricSpaces}]
We have $\log_{4/3}(2n)+7\geq 2 \log_{2}(2n)+7$ for $n\geq 1$. By Theorem 3.3 in \cite{Kennard2}, there exists a product of spheres $S = \s^{n_1}\times\cdots\times\s^{n_t}$ with $n_i\geq 16$ such that
	\begin{enumerate}
	\item $N = S$,
	\item $N = S\times R$ where $R\in\{\cc\pp^m, \SO(m+2)/\SO(m)\times\SO(2)\}$,
	\item $N = S\times R$ where $R\in\{\hh\pp^m, \SO(m+3)/\SO(m)\times\SO(3)\}$,
	\item $N = S\times R\times \s^2$ where $R\in\{\hh\pp^m, \SO(m+3)/\SO(m)\times\SO(3)\}$, or
	\item $N = S \times R\times \s^3$ where $R\in\{\hh\pp^m, \SO(m+3)/\SO(m)\times\SO(3)\}$.
	\end{enumerate}
Note that the number of spherical factors $s$ (as in the statement of the corollary) is $t$ in the first three cases and is $t+1$ in the last two cases.

It follows that $b_2(M) = 0$ or $b_3(M) = 0$. In particular, $\chi(M) > 0$ by Theorem \ref{thm:OddBettis}. This implies that each $n_i$ is even, that Case (5) cannot occur, and that $R = \SO(m+3)/\SO(m)\times\SO(3)$ can only occur if $m$ is even.

In Cases (1), (2), and (3), we estimate the Euler characteristic by
	\[\chi(N) \geq \chi(S) = 2^t = 2^s,\]
and, in Case (4), we estimate
	\[\chi(N) \geq \chi(S)\chi(\s^2) = 2^{t+1} = 2^s.\]
Since $\chi(N) = \chi(M) < 2^{3(\log_2 n)^2}$, we conclude that the number of spherical factors $s$ is less than $3(\log_2 n)^2$.

For this we make use of
\begin{align*}
\chi(\SO(2+m)/\SO(2)\times\SO(m))&=m+2=n+2 \qquad \textrm{for even $m$}\\
\chi(\SO(2+m)/\SO(2)\times\SO(m))&=m+1=n+1 \qquad \textrm{for odd $m$}\\
\chi(\SO(3+m)/\SO(3)\times\SO(m))&=m+2=2n/3+2  \qquad \textrm{for even $m$}
\end{align*}
and $\dim \SO(p+m)/\SO(p)\times\SO(m)=pm$ where $2n$ denotes the respective dimensions of the manifolds. (We recall that the Euler characteristics of the homogeneous spaces can easily be computed using the classical formula $\chi(G/H)=|\W(G)|/|\W(H)|$, i.e.~as the quotient of the cardinalities of the respective Weyl groups.)
\end{proof}

We leave it to the interested reader to vary this result using the factors $\kappa_i$ from Section \ref{ProofOfLogAssumption}. Besides, a variation of this result using Theorem \ref{LinearAssumptionDelta} is obvious. Also, the presented result is certainly only interesting for large $n$, since in small dimensions dimension estimates will be better.


\pagebreak

\
\vfill

\begin{center}
\noindent
\begin{minipage}{\linewidth}
\small \noindent \textsc
{Manuel Amann} \\
\textsc{Fakult\"at f\"ur Mathematik}\\
\textsc{Institut f\"ur Algebra und Geometrie}\\
\textsc{Karlsruher Institut f\"ur Technologie}\\
\textsc{Kaiserstra\ss e 89--93}\\
\textsc{76133 Karlsruhe}\\
\textsc{Germany}\\
[1ex]
\textsf{manuel.amann@kit.edu}\\
\textsf{http://topology.math.kit.edu/$21\_54$.php}
\end{minipage}
\end{center}

\vspace{10mm}

\begin{center}
\noindent
\begin{minipage}{\linewidth}
\small \noindent \textsc
{Lee Kennard} \\
\textsc{Department of Mathematics}\\
\textsc{University of California}\\
\textsc{Santa Barbara, CA 93106-3080}\\
\textsc{USA}\\
[1ex]
\textsf{kennard@math.ucsb.edu}\\
\textsf{http://www.math.ucsb.edu/$\sim$kennard/}
\end{minipage}
\end{center}

\end{document}